\documentclass[10pt,reqno]{amsart}

\usepackage{amsmath,mathrsfs}
\usepackage{graphicx,color}

\usepackage[colorlinks=true,linkcolor=blue,citecolor=red]{hyperref}

\numberwithin{equation}{section}

\def\C{\ensuremath{\mathbb{C}}}
\def\E{\ensuremath{\mathbb{E}}}
\def\N{\ensuremath{\mathbb{N}}}

\def\R{\ensuremath{\mathbb{R}}}

\def\Z{\ensuremath{\mathbb{Z}}}

\newcommand{\mr}[1]{\mathrm{#1}}

\newcommand{\cF}{\ensuremath{\mathcal{F}}}
\newcommand{\cH}{\ensuremath{\mathcal{H}}}

\newcommand{\abs}[1]{\left|{#1}\right|}
\newcommand{\nrm}[1]{\left\|{#1}\right\|}
\newcommand{\set}[1]{\left\{#1\right\}\relax}
\newcommand{\scal}[1]{\left\langle\relax #1 \relax\right\rangle}
\newcommand{\qtxtq}[1]{\quad \text{#1}\quad}

\newtheorem{thm}{Theorem}[section]

\newtheorem{prop}[thm]{Proposition}
\newtheorem{lem}[thm]{Lemma}
\newtheorem{cor}[thm]{Corollary}
\newtheorem{example}[thm]{Example}
\newtheorem{rem}[thm]{Remark}

\usepackage{concmath}

\title{A Law Limit Theorem for a sequence of random variables}

\author{M. Maslouhi}
\address{M. Maslouhi:
IbnTofail university. Kenitra 14000. Morocco.
}

\email {mostafa.maslouhi@uit.ac.ma}



\begin{document}

\begin{abstract} 
An application  of Levy's continuity theorem and Hankel transform allow us to establish a law limit theorem for the sequence  $V_n=f(U)\sin(n U)$, where $U$ is uniformly distributed in $(0,1)$ and $f$ a given function. 
 Further, we investigate the inverse problem by  specifying a limit distribution and look for the suitable function $f$ ensuring the convergence in law to the specified distribution. Our work recovers and extends existing similar works, in particular we  make it possible to sample from  known laws including  Gaussian and Cauchy  distributions.
\end{abstract}

\maketitle

\markboth{M. Maslouhi}{A Law Limit Theorem for a sequence of random variables}

\section{Introduction and preliminaries}\label{sec:intro}

Studying the asymptotic behavior of sequences of random variables has always been crucial and a  fundamental tool for theoretical research and practical applications in fields such as probability, statistics, finance and engineering  to mention few. See \cite{montgomery2019introduction,robert2014machine,hastie2009elements,black1973pricing} for more details. In these areas,  the Law of Large Numbers and the Central Limit Theorems play crucial role. 
In recent years, due to the huge development in communication and information fields, there is a growing need for simulations using specific distributions in a wide range of disciplines. See \cite{lee2004gaussian,box1958note,thomas2007gaussian,toral1993generation,schollmeyer1991noise,marsaglia2000ziggurat} and the references therein.
This paper aims to follow in this direction and  explore a law limit theorem for a specific sequence of random variables given by \eqref{eq:def_V_n}. This exploration uses two key tools, namely Levy's continuity theorem together with Hankel transform and yields two main results. The first  main result gives sufficient conditions on the parameter function $f$  under which the convergence in law of   $V_n[f]$ holds. (See Theorem \ref{thm:main_result}). The second main result  investigates the inverse problem.  
More precisely, starting from a suitable given density, we give sufficient conditions allowing us to construct a function $f$ such that the sequence $V_n[f]$ given by \eqref{eq:def_V_n} converges in law to our specified limiting distribution (See Theorem \ref{thm:solving_PDF}).
 
We point out here that a  version of Theorem \ref{thm:main_result} has been established in  \cite{addaim2018enhanced} in the particular case $f(u)=\sqrt{-2\ln(u)}$ which led to the normal distribution using direct calculations. Our work recover their result by exploring a general framework and giving the convergence problem a wide perspective.  Indeed, by investigating this  general problem, the direct and inverse one, we can  obtain a convergence in law to a wide class of districbutions including the Cauchy distribution as shown in this paper.  

The remaining of this section serves to introduce some notations and definitions  needed for the sequel. Section \ref{sec:main_result} sets our  main results, namely Theorem  \ref{thm:main_CV}, Theorem \ref{thm:main_result} and Theorem \ref{thm:solving_PDF}. In Section \ref{sec:applications} we apply our results  by giving some examples  to illustrate the practical aspect  of our work.  Finally, in Section  \ref{sec:conclusion}, we summarize our findings and suggest some directions for future work.

Throughout this paper, $\N=\set{0,1,2,\dots}$ , $\Z=\set{\dots,-2,-1,0,1,2,\dots}$   and $\R^d$ is equipped with its standard scalar product and Euclidean norm, denoted in the sequel respectively by $\scal{,}$ and $\nrm{\ }$.

The sets $L^1(\R^d)$ and $ L^1_{loc}((a,b))$, $(a,b)\subset \R$, keep their standard meaning in this paper with respect to Lebesgue measure.

For $\psi:\R^d\to\C$, with $\psi\in L^1(\R^d)$, $d\geq1$, its Fourier transform is given by
$$\cF_d(\psi)(y)=\frac{1}{(2\pi)^{d/2}}\int_{\R^d} \psi(x)e^{-i\scal{x,y}}dx.$$

Following \cite{temme2011special}, Bessel function of the first kind and order $n\in \Z$, denoted hereafter by $J_{n}$ is given by its integral expression
\begin{equation}\label{eq:J_n_def_integral}
 J_{n}(x) ={\frac {1}{2\pi }}\int _{-\pi }^{\pi }e^{i(x\sin(u)- nu)}\,du,\quad x\in\R.
\end{equation}
See \cite{abramowitz1948handbook, watson1922treatise} for details on Bessel functions.

In particular, we see from \eqref{eq:J_n_def_integral} that for all $n\in\Z$, $J_{n}$ is bounded in $\R$ and 
\begin{equation}\label{eq:J_n_relationship}
	J_{-n}(x)=(-1)^nJ_{n}(x),\quad  \forall x\in\R,
\end{equation}
and that the function $J_0$,  which plays an important role in this paper, is even.  
  Morover, for all $n\in\N$, $J_n$ admits the series expansion 
\begin{equation}\label{eq:J_n_series_expr}
	J_{n}(x)=\sum_{k=0}^{\infty }\frac {(-1)^{k}}{k!(k+n)!}\left({\frac {x}{2}}\right)^{2k+n},\quad  x\in\R.
\end{equation}
According to Watson \cite[Page 24]{watson1922treatise}  we have

\begin{equation}\label{eq:J_n_bound}
	\abs{J_{n}(x)}\leq\frac {\abs{x}^{n}}{2^n\Gamma(n+1/2)\Gamma(1/2)}
\end{equation}
for all $x\in\R$ and $n\in\N$.

The Hankel transform of order 0 of a function $f:(0,+\infty)\to\R$ is given by:
\begin{equation}\label{eq:def_Hankel_Trans}
	\mathcal{H}_{0}(f)(t)=\int _{0}^{\infty }f(r)J_{0}(rt)\,r\operatorname{d}\!r,\quad t\in\R
\end{equation}
provided that  $x\mapsto xf(x)\in L^1((0,+\infty))$. Following \cite[Page 456]{watson1922treatise} ( See also \cite[Theorem 19]{sneddon1972use}), if $f$ is further of bounded variation in each bounded sub-interval of $(0,+\infty)$ and $x\mapsto \sqrt{x}f(x)\in L^1((0,+\infty))$, then  $f=\cH_0(\cH_0(f))$. That is
\begin{equation}\label{eq:Hankel_inverse}
f(r)=\int _{0}^{\infty }\mathcal{H}_{0}(f)(t)J_{0}(rt)\,t\operatorname{d}\!t, \quad r\geq 0.
\end{equation}

In \cite{poularikas2010transforms}, a relationship between the Fourier transform of radial functions in $\R^2$ and their Hankel transform is given.
\begin{thm}\cite{poularikas2010transforms}\label{thm:radial_FT}.\quad 
	Let  $w\in L^1(\R^2)$ be  radial  with $w(x)=G(\nrm{x})$. Then for all $y\in\R^2$ we have
	$$\cF_2(w)(y)=\mathcal{H}_{0}(G)(\nrm{y}).$$ 
\end{thm}

\medskip

 One precious tool in this paper is Levy continuity's theorem (see \cite{fristedt2013}), which turns to be the  corner stone in the proof of Theorem \ref{thm:main_result}.
 
 \begin{thm}[L\'evy's continuity Theorem]\label{thm:Levy}
 	Consider a sequence of random variables $X_{n}$, $n\ge 1$, and define the sequence of corresponding characteristic functions $$\varphi_n(t):=\E\left(e^{itX_{n}}\right),\quad (t\in\R),$$
 	where $\E$ is the expected value operator. If the sequence $\varphi_n$ converges point wise to some function $\varphi$ which is continuous at $t=0$,  then the sequence $X_{n}$ converges in distribution to some random variable $X$ and $\varphi$ is the characteristic function of the random variable $X$. 
 \end{thm}

The next result, useful for its own, turns out to be crucial for our developments. Up to our knowledge the formula \eqref{eq:fundamental_equat} exists nowhere. 
\begin{prop}\label{pro:on_bessel_expansion}
	Let $w\in\R$. Then we have
	\begin{equation}\label{eq:h_Fourier_series}
		e^{iw\sin(x)}=\sum_{k=-\infty}^{\infty}J_k(w)e^{ik x}
	\end{equation}
	for all $x\in\R$. 	Moreover,  \begin{equation}\label{eq:fundamental_equat}
		\frac{w^2}{2}=\sum_{k=-\infty}^{\infty}k^2(J_k(w))^2
	\end{equation}
	
\end{prop}
\begin{proof}
	For $w\in\R$ fixed, define  the function
	$$h(x):=e^{iw\sin(x)},\quad x\in \R.$$  Keeping the formula  \eqref{eq:J_n_def_integral} in mind, the Fourier series expansion  of $h$ gives us 
	\begin{equation*}
		h(x)=\sum_{k=-\infty}^{\infty}J_k(w)e^{ik x}
	\end{equation*}
	for all $x\in\R$ and this proves our first claim. 
	
	In other hand,  from \eqref{eq:J_n_relationship}  and \eqref{eq:J_n_bound} we have  $\sum_{k=-\infty}^{\infty}\abs{kJ_k(w)}<+\infty$, which allows us  to derive with respect to $x$ in \eqref{eq:h_Fourier_series} and get 	
	$$w\cos(x)e^{iw\sin(x)}=\sum_{k=-\infty}^{\infty}kJ_k(w)e^{ik x}$$ for all $x\in\R$. Using Parseval identity, we deduce that
	$$\frac{w^2}{2\pi}\int_{-\pi}^{\pi}\abs{\cos(x)}^2dx=\sum_{k=-\infty}^{\infty}k^2(J_k(w))^2$$
which is the second desired result.
	
\end{proof}

Before leaving this section we establish  a well known simple lemma  needed for the sequel.
\begin{lem}\label{lem:Fourier_transf}
	Fix $a>0$ and set  $$\theta_a(t)=\frac{1}{t^2+a^2},\quad t\in\R.$$  Then we have  
	$\cF_1(\theta_a)(t)= \frac{\sqrt{\pi}}{a\sqrt{2}}e^{-a\abs{t}}$ for all $t\in\R$.
\end{lem}
\begin{proof}
	Noting that  $$\cF_1(e^{-a\abs{u}})(t)=a\sqrt{\frac{2}{\pi}}\theta_a(t)$$ the result follows using the inverse Fourier transform.  	
\end{proof}

Now we have all needed ingredients for our developments.
 \section{Main Results} 
 \label{sec:main_result}
 
 We begin our developments by the first main result in this paper. It is the corner stone on which all the other results are constructed. 
 \begin{thm}\label{thm:main_CV}
 	Consider $f:(a,b)\subset \R\to \R$, with $f\in L^1((a,b))$. 
 	Then  we have $$\lim_{n\to\infty}\int_{a}^{b}e^{if(u)\sin(nu)}du=\int_{a}^{b}J_0(f(u))du.$$ 
 	Further, if $a,b\in\R$ then the result holds with $f\in L^1_{loc}(a,b)$ only.
 \end{thm}
 \begin{proof}Let  $u\in (a,b)$.  	Taking $w=f(u)$ in Proposition \ref{pro:on_bessel_expansion} and letting $x=nu$ in \eqref{eq:h_Fourier_series} we get
 	$$e^{if(u)\sin(n u)}=\sum_{k=-\infty}^{\infty}J_k(f(u))e^{iknu}.$$  
 	Whence, 
 	\begin{align*}
 		\int_{a}^{b} e^{if(u)\sin(n u)}du-\int_{a}^{b} J_0(f(u))du =\int_{a}^{b}\left(\sum_{{}^{k=-\infty}_{k\neq 0}}^{\infty}J_k(f(u))e^{iknu}\right)du
 	\end{align*}
 	
 	We assume first here that $f\in L^1((a,b))$.
 	Using Cauchy-Schwartz  inequality together with \eqref{eq:fundamental_equat}, yields 
 	\begin{align}
 		&\int_{a}^{b}\sum_{{}^{k=-\infty}_{k\neq 0}}^{\infty}\abs{J_k(f(u))}du
 		\leq C_1\int_{a}^{b}\left(\sum_{{}^{k=-\infty}_{k\neq 0}}^{\infty}(kJ_k(f(u)))^2\right)^{1/2}\leq C_2 \int_{a}^{b}\abs{f(u)}du\label{ali:cvu_Rieman_Lebesgue}
 	\end{align}
 	where $C_1=\left(2\sum_{k=1}^{\infty}\frac{1}{k^2}\right)^{1/2}$ and $C_2=\frac{C_1}{\sqrt{2}}$. 
 	
 		Taking $f\in L^1((a,b))$ into account allows us to write  
 	\begin{align*}
 		\int_{a}^{b} e^{if(u)\sin(n u)}du-\int_{a}^{b} J_0(f(u))du =\sum_{{}^{k=-\infty}_{k\neq 0}}^{\infty}\int_{a}^{b} J_k(f(u))e^{iknu}du.
 	\end{align*}
 	Finally, by Riemann-Lebesgue lemma together with  \eqref{ali:cvu_Rieman_Lebesgue}, we get that  	$$\lim_{n\to\infty}\int_{a}^{b}e^{if(u)\sin(nu)}du=\int_{a}^{b}J_0(f(u))du.$$

 	Assume now that $a,b\in \R$ and  $f\in L^1_{loc}(a,b)$. Fix $\epsilon>0$ and choose $\alpha,\beta$ such that $$\abs{a-\alpha}<\epsilon \qtxtq{and}\abs{b-\beta}<\epsilon.$$ Then for all $n\in\N$ we have
 	$$\abs{\int_{a}^{b}e^{if(u)\sin(nu)}du-\int_{a}^{b}J_0(f(u))du}\leq\abs{\int_{\alpha}^{\beta}e^{if(u)\sin(nu)}du-\int_{\alpha}^{\beta}J_0(f(u))du} +4\epsilon.$$ Our result follows now using the previous case and the proof is complete. 
 \end{proof}

  In the sequel, $U$ denotes a random variable uniformly distributed in the interval $[0,1]$. For  $f:(0,1)\to \R$  define the random variable  $V_n[f]$ given by 
  \begin{equation}\label{eq:def_V_n}
  	V_n=V_n[f]:=f(U)\sin(n U),\quad n=0,1,2,\dots
  \end{equation}  

An application of Theorem \ref{thm:main_CV} and Levy's theorem \ref{thm:Levy} all together lead to our second main result.
\begin{thm}\label{thm:main_result}
	Assume that the function $f$ in \eqref{eq:def_V_n} satisfies $f\in L^1_{loc}(0,1)$.
	
Then the sequence of random variables $V_n[f]$ defined by \eqref{eq:def_V_n} converges in law to a random variable $V$ such that $$\E(e^{itV})=\int_{0}^{1}J_0(tf(u))du, \quad \forall t\in \R.$$
\end{thm}
\begin{proof}
Since $f\in L^1_{loc}((0,1))$ then by application of Theorem \ref{thm:main_CV} we get 
$$\lim_{n\to\infty}\E(e^{itf(U)\sin(n U)})=\lim_{n\to\infty}\int_{0}^{1}e^{itf(u)\sin(nu)}du=\int_{0}^{1}J_0(tf(u))du$$ for all $t\in\R$. 
The function $\varphi$ defined \begin{equation}\label{eq:def_Phi}
	\varphi(t):=\int_{0}^{1}J_0(tf(u))du, \quad t\in \R
\end{equation}
 is continuous, by consequence appealing to Levy Theorem \ref{thm:Levy} there exists a random variable $V$ such that $V_n[f]$ converges in law to $V$ and $$\E(e^{itV})=\int_{0}^{1}J_0(tf(u))du, \quad t\in \R,$$
 and the proof is complete.	
\end{proof}

The next result gives sufficient conditions on the function $f$ so that the random variable $V$  in Theorem \ref{thm:main_result} admits a  probability density function. 
\begin{cor}\label{cor:sufficient_cond}
	We reuse the notations of Theorem \ref{thm:main_result} and assume that    
	 $f$ is a strictly monotone bijection of class $C^1$  from $(0,1)$ into some $(a,b)\subset (0,+\infty)$ such that $f\in L^1_{loc}(0,1)$. 	 Let 
	 $$h(u):=\epsilon_f\chi_{(a,b)}(u)\frac{(f^{-1})^{\prime}(u)}{u},\quad u\in(0,+\infty)$$ with $\chi_{(a,b)}$ being the characteristic  function of the set $(a,b)$, and     $\epsilon_f=1$ if $f$ is increasing and $\epsilon_f=-1$ if $f$ is  decreasing.

	  Assume that 	  	
	  	the map $v:=\cH_0(h)$  satisfies 	 
	  	$v\in L^1(\R)$.     
	Then the sequence of random variables $V_n[f]$ defined by \eqref{eq:def_V_n} converges in law to a random variable $V$ having  $\frac{1}{\sqrt{2\pi}}\cF_1(v)$ as a probability density  function. 
\end{cor}
\begin{proof}
	From Theorem \ref{thm:main_result} we have   
	$$\E(e^{itV})=\epsilon_f\int_{a}^{b}J_0(tu) (f^{-1})^{\prime}(u)du=\int_{0}^{\infty}J_0(tu) h(u)udu=v(t)$$ for all $t\in\R$. 	
	By hypothesis  $v\in L^1(\R)$, thus   the probability density function of $V$ is $\frac{1}{\sqrt{2\pi}}\cF_1(v)$ and this ends the proof.
\end{proof}

\begin{rem}\label{rem:using_F_2}
	Keeping Corollary \ref{cor:sufficient_cond} notations in mind, it is worth noting that if we assume     
$h(\nrm{\cdot})\in L^1(\R^2)$, then by Theorem \ref{thm:radial_FT} we get   $v(t)=\cF_2(h(\nrm{\cdot}))(\abs{t}e_1)$ for all $t\in\R$ where $e_1=(1,0)$. 	 
\end{rem}

Now we come to our third main result in this paper which we will refer to by the inverse problem. Our  objective here is to give sufficient conditions so that the inverse problem admits a solution. More precisely, given a suitable map $\psi\in L^1(\R)$, we give
 sufficient conditions  ensuring that there exists a function $f$ such that the sequence $V_n[f]$ given by \eqref{eq:def_V_n} converges in law to a random variable $V$ having $\frac{1}{\sqrt{2\pi}}\cF_1(\psi)$ as probability density function. 
Solving this problem is particularly useful when one wants to sample some data from a given none standard distribution.  See  Section \ref{sec:applications} for a an example of the Cauchy distribution using this inverse problem.

Before stating our result,  it is worth noting that following Theorem \ref{thm:main_result}, there are some conditions that have to be fulfilled by the  function $\psi$. This leads us to investigate our inverse problem within the class of functions $\psi:\R\to\R$ satisfying the following, not restrictive, properties denoted hereafter by \eqref{eq:def_L}.
\begin{equation}\tag{$\mathcal{L}$}\label{eq:def_L}
 \left\{
	\begin{array}{ll}
		1. & \psi\in C^1(\R),\\[3pt] 
			2.&\psi, \sqrt{t}\psi(t)  \text{ and }  t\psi(t) \text{ are in }  L^1((0,\infty)),\\[3pt] 
			3.& \psi  \text{ even,  non negative with }  \psi(0)=1.
	\end{array}\right.
\end{equation}

For a function  $\psi$   satisfying \eqref{eq:def_L} we set
\begin{equation}\label{eq:def_k_psi}
	k_{\psi}(t):=1-\int_{0}^{t}u\cH_0(\psi)(u)du,\quad t\geq0.
\end{equation}
Appealing to \eqref{eq:def_L} and \eqref{eq:Hankel_inverse} we see that $k_\psi$ is a decreasing bijection from   $(0,+\infty)$ into $(0,1)$. 
Moreover, we have the following:

\begin{thm}\label{thm:solving_PDF}
	Let $\psi$ be a function satisfying \eqref{eq:def_L} such that $(k_\psi)^{-1}\in L^1_{loc}(0,1)$. 
 Then the sequence  of random  variables $V_n[(k_{\psi})^{-1}]$ given by \eqref{eq:def_V_n}   converges  in law to a random  variable $V$ having $\frac{1}{\sqrt{2\pi}}\cF_1(\psi)$ as probability density function.
\end{thm}
\begin{proof} Define the sequence  of random  variables $V_n[f]$ as in \eqref{eq:def_V_n} with $f(t)=(k_{\psi})^{-1}(t)$, $t\in (0,1)$. 	
	From \ref{thm:main_result},  we get that the sequence $V_n[f]$
	converges  in law to a random  variable $V$ with $$\E(e^{itV})=-\int_{0}^{\infty}J_0(tu) (f^{-1})^{\prime}(u)du=-\int_{0}^{\infty}J_0(tu) k_\psi^{\prime}(u)du,$$ for all $t\in\R$. 		
	Using the expression \eqref{eq:def_k_psi} leads to 
	$$\E(e^{itV})=\cH_0(\cH_0(\psi))(t),$$ for all $t\in\R$. Since $\psi$ satisfies \eqref{eq:def_L} the  Hankel transform inversion formula \eqref{eq:Hankel_inverse}  gives us  $$\cH_0(\cH{_0(\psi)})(t)=\psi(t),\quad  \forall t\geq 0.$$
	Taking into account that $\psi$ is even, the latter equality holds for all $t\in\R$, that is 
	$$\E(e^{itV})=\psi(t),\quad  \forall t\in\R.$$ Thus   $\frac{1}{\sqrt{2\pi}}\cF_1(\psi)$ is the probability density function for $V$ and the proof is complete.	
\end{proof}

\section{Applications} 
\label{sec:applications}

This section is devoted to some applications using the results and notations of Section \ref{sec:main_result}. Namely, we will use two kind of applications. 
 The first one, which we will call ``Direct problem", consists of starting from a given function $f$ and use the results of Theorem \ref{thm:main_result} or Corollary \ref{cor:sufficient_cond}. The second application, which we will call ``Inverse problem", will starts from a given map $\psi$  and use the results of Theorem \ref{thm:solving_PDF} to construct the function $f$ in \eqref{eq:def_V_n}.
 
\begin{example}\label{exmpl:direct_Gaussian} In this example we use the notations and results of Corollary \ref{cor:sufficient_cond} and
  		consider $f$ given by $$f(x)=\sqrt{-2\ln(x)},\quad x\in(0,1).$$
  		We see that $f\in L^1((0,1))$, and  is a decreasing bijection of class $C^1$ from $(0,1)$ to $(0,+\infty)$. Moreover, by setting  $$h(u):=-\frac{(f^{-1})^{\prime}(u)}{u}=e^{-u^2/2},\quad u>0$$ direct application of Corollary \ref{cor:sufficient_cond}  leads to 
  		\begin{align*}
  			\psi(t)=e^{-t^2/2}, \quad t\in\R
  		\end{align*}
  		 and t that the sequence $V_n[f]$ defined by 
$$V_n=\sqrt{-2\ln(U)}\sin(n U),\quad n\in\N$$ converges in law to the normal Gaussian. Thus, we  recover the result in \cite{addaim2018enhanced}.   		
\end{example}

\begin{example}\label{exmpl:inverse_Gaussian} 
	
In this example we illustrate the technique of  inverse problem mentioned above for two given distributions.
\begin{enumerate}
	\item 
	We consider the Gaussian distribution density  $\psi(t)=e^{-t^2/2}$, $t\in\R$. 	
	Following Theorem \ref{thm:solving_PDF}, the candidate function $f:=(k_{\psi})^{-1}$ is solution of 
	$$k_{\psi}^{\prime}(t)=-t\cH_0(\psi)(t),\quad t>0.$$
	
	According to  
	\cite[Table 9.1]{poularikas2010transforms} we have
	$$k_{\psi}^{\prime}(t)=-te^{-t^2/2},\quad t>0$$ from which it follows that
	$$f(u)=(k_{\psi})^{-1}(u)=\sqrt{-2\ln(u)}, \quad u\in (0,1).$$
	Thus, we recover again   the result in \cite{addaim2018enhanced} using the inverse problem. 
	\item 	
	Here the inverse method problem us to construct a sequence $V_n[f]$ as  in \eqref{eq:def_V_n} that converges in law to a Cauchy distribution. This distribution is well known for its key role in several scientific areas.  See  \cite{feller1991introduction,cardona2007light,tyler2008robust,dimitris2000statistical}  for details and applications of  this distribution.
	
	\medskip

	Let $\theta(t):=\frac{1}{t^2+\pi/2}$ and  $\psi=\cF_1(\theta)$, $t\in\R$. 	
	 From  Lemma \ref{lem:Fourier_transf} 
	 and using  
 \cite[Table 9.1]{poularikas2010transforms} we get
	$$
		\cH_0(\psi)(t)=\frac{\sqrt{\pi/2}}{(t^2+\pi/2)^{3/2}},\quad \forall t\in\R.
$$
	Following the notations in Theorem \ref{thm:solving_PDF}, we get
	$$k_{\psi}(t)=1-\sqrt{\pi/2}\int_{0}^{t}\frac{u}{(u^2+\pi/2)^{3/2}}du=\frac{\sqrt{\pi}}{\sqrt{2t^2+\pi}}$$ 
 	 and then 
	\begin{equation}\label{eq:def_f_Cauchy}
		f(u):=(k_{\psi})^{-1}(u)=\sqrt{\frac{\pi}{2}}\frac{\sqrt{1-u^2}}{u},\quad u\in (0,1).
	\end{equation}
	Using Theorem \ref{thm:solving_PDF}, we deduce that the sequence $V_n[f]$ in \eqref{eq:def_V_n} with $f$ given by \eqref{eq:def_f_Cauchy} converges in law to a random variable having $\frac{1}{\sqrt{2\pi}}\cF_1(\psi)=\frac{1}{\sqrt{2\pi}}\theta$ as probability density function, which is nothing but $\mr{Cauchy}(0,\sqrt{\pi/2})$.
	 Figure \ref{fig:cauchy_distrib} shows a sampling of the random variable $V_n$ with  $n=1000$ and 10000 samples of the uniform distribution in (0,1).
\end{enumerate}

\begin{figure}[h!]
	\centering
	\includegraphics[scale=.4]{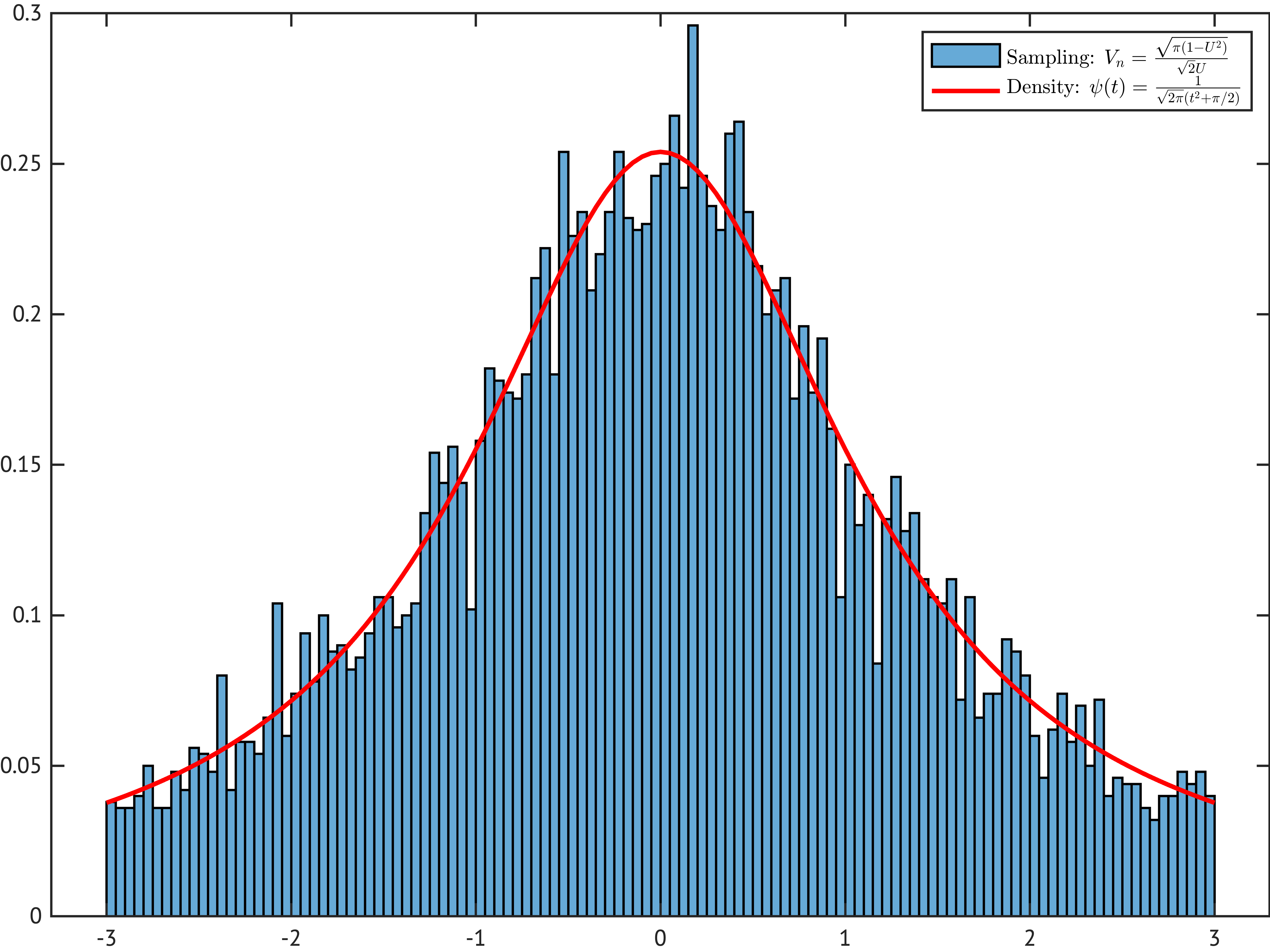}
	\caption{Cauchy distribution}
	\label{fig:cauchy_distrib}
\end{figure}

\end{example}

\section{Conclusion}\label{sec:conclusion}	
The results in Theorem \ref{thm:main_result}, Corollary \ref{cor:sufficient_cond}  allow one to sample and simulate data when  the parameter  function $f$ is known.  More interestingly, the inverse problem solved by \ref{thm:solving_PDF} may be initiated to construct the ``unknown" function $f$ in \eqref{eq:def_V_n} which allows one to sample from a given particular distribution. 

Some remarks to conclude our work are as follows:
Among future perspectives for our work we may cite:
\begin{enumerate}
	\item It is not hard to extend the results in  Theorem \ref{thm:main_result} and Theorem \ref{thm:solving_PDF} to  random variables with several variables of the form $$V_n[f]=(f_1(U_1)\sin(n U_1),\dots,f_d(U_d)\sin(n U_d))$$ where $f=(f_1,\dots,f_d)$, $U_j$  uniformly distributed in $(0,1)$ and $f_j\in L^1(0,1)$.  
	\item The condition \eqref{eq:def_L}.1) may be  relaxed as ``$\psi$ of bounded variation in each  $[a,b]\subset (0,\infty)$". See \cite{sneddon1995fourier}.
	\item A large class of functions satisfying \eqref{eq:def_L} is the set of even functions $\psi$ such that $\psi\in L^1(0,\infty)$ and  $\psi(t)=\sqrt{t}\varphi(t)$ with $\varphi$ is in the Schwartz class  $\mathcal{S}(0,\infty)$. 
\end{enumerate}

In a future work the following problems will be investigated.
\begin{enumerate}
	\item Study the stability of the limiting law for the direct problem, respectively the stability of the parameter function $f$ for the inverse problem.  That is:	
	\begin{itemize}
		\item If $f$ is replaced by a perturbed version  $f_{\epsilon}$, how the limiting law is perturbed?  
		\item Respectively, if the density $\psi$ is replaced by a perturbed version  $\psi_{\epsilon}$, how the parameter function $f$ is perturbed?  
	\end{itemize}

	\item Looking for characterization results in  Theorem \ref{thm:main_result} and Theorem \ref{thm:solving_PDF}.

\end{enumerate}

\bibliographystyle{abbrv}
\bibliography{rv2pdf.bib}
\end{document}